\documentclass[12pt,a4paper,english,reqno]{amsart}
\usepackage[a4paper,footskip=1.5em]{geometry}
\usepackage{amsmath,amssymb,amsthm,mathtools,bbm}
\usepackage[mathscr]{euscript}
\usepackage[usenames,dvipsnames]{color}
\usepackage{adjustbox,tikz,calc,graphics,babel,standalone}
\usepackage{subcaption}
\usepackage{csquotes,enumerate,verbatim}
\usepackage[final]{microtype}
\usepackage[numbers]{natbib}
\usetikzlibrary{shapes.misc,calc,intersections,patterns,decorations.pathreplacing}
\usepackage{hyperref}
\hypersetup{colorlinks=true,linkcolor=blue,citecolor=blue,pdfpagemode=UseNone,pdfstartview={XYZ null null 1.00}}
\usepackage{cmtiup}
\usepackage{todonotes}

\theoremstyle{plain}
\newtheorem*{theorem*}{Theorem}
\newtheorem{theorem}{Theorem}[section]

\newtheorem{claim}[theorem]{Claim}
\newtheorem{proposition}[theorem]{Proposition}
\newtheorem*{claim*}{Claim}

\newtheorem{problem}[theorem]{Problem}

\theoremstyle{remark}

\DeclareMathOperator{\bin}{Bin}

\def\Cond{\,|\,}

\newcommand{\Z}{\mathbb{Z}}
\newcommand{\R}{\mathbb{R}}
\newcommand{\E}{\mathbb{E}}

\newcommand{\N}{\mathbb{N}}
\renewcommand{\P}{\mathbb P}
\newcommand{\xx}{\mathbf{x}}
\newcommand{\yy}{\mathbf{y}}
\newcommand{\WW}{\mathcal{W}}

\DeclareMathOperator\Bin{Bin}

\let\eps\varepsilon
\let\originalleft\left
\let\originalright\right
\renewcommand{\left}{\mathopen{}\mathclose\bgroup\originalleft}
\renewcommand{\right}{\aftergroup\egroup\originalright}

\begin{document}

\title{Exceptional graphs for the random walk}

\author{Juhan Aru}
\address{Departement Mathematik, ETH Z\"urich, R\"amistrasse 101, 8092, Z\"urich, Switzerland}
\email{juhan.aru@math.ethz.ch}

\author{Carla Groenland}
\address{Mathematical Institute, University of Oxford, Andrew Wiles Building, Radcliffe Observatory Quarter, Woodstock Road, Oxford OX2\thinspace6GG, UK}
\email{groenland@maths.ox.ac.uk}

\author{Tom Johnston}
\address{Mathematical Institute, University of Oxford, Andrew Wiles Building, Radcliffe Observatory Quarter, Woodstock Road, Oxford OX2\thinspace6GG, UK}
\email{johnston@maths.ox.ac.uk}

\author{Bhargav Narayanan}
\address{Department of Mathematics,	Rutgers University, Piscataway NJ 08854, USA}
\email{narayanan@math.rutgers.edu}

\author{Alex Roberts}
\address{Mathematical Institute, University of Oxford, Andrew Wiles Building, Radcliffe Observatory Quarter, Woodstock Road, Oxford OX2\thinspace6GG, UK}
\email{roberts@maths.ox.ac.uk}

\author{Alex Scott}
\address{Mathematical Institute, University of Oxford, Andrew Wiles Building, Radcliffe Observatory Quarter, Woodstock Road, Oxford OX2\thinspace6GG, UK}
\email{scott@maths.ox.ac.uk}

\date{20 July 2018}


\begin{abstract}
If $\WW$ is the simple random walk on the square lattice $\Z^2$, then $\WW$ induces a random walk $\WW_G$ on any spanning subgraph $G\subset \Z^2$ of the lattice as follows: viewing $\WW$ as a uniformly random infinite word on the alphabet $\{\xx, -\xx, \yy, -\yy \}$, the walk $\WW_G$ starts at the origin and follows the directions specified by $\WW$, only accepting steps of $\WW$ along which the walk $\WW_G$ does not exit $G$. For any fixed $G \subset \Z^2$, the walk $\WW_G$ is distributed as the simple random walk on $G$, and hence $\WW_G$ is almost surely recurrent in the sense that $\WW_G$ visits every site reachable from the origin in $G$ infinitely often. This fact naturally leads us to ask the following: does $\WW$ almost surely have the property that $\WW_G$ is recurrent for \emph{every} $G \subset \Z^2$? We answer this question negatively, demonstrating that exceptional subgraphs exist almost surely. In fact, we show more to be true: exceptional subgraphs continue to exist almost surely for a countable collection of independent simple random walks, but on the other hand, there are almost surely no exceptional subgraphs for a branching random walk.
\end{abstract}
\maketitle

\section{Introduction}
Let us say that a walk on a graph $G$ is \emph{recurrent} if the walk visits every site in the connected component of its starting point in $G$ infinitely often, and \emph{transient} otherwise. It is a classical result of Poly\'a~\citep{Pol}  that a simple random walk on the square lattice $\Z^2$ is almost surely recurrent. In this paper, we shall be concerned with how `robust' this property is in the following sense: do the coin tosses that determine a recurrent random walk on $\Z^2$ also determine a recurrent random walk on \emph{every} subgraph of $\Z^2$ simultaneously? We make this question precise below in a few different ways.

We view a simple random walk $\WW$ on $\Z^2$ as a random infinite word on the four-letter alphabet $\{\xx, -\xx, \yy, -\yy \}$, where $\xx = (1,0)$ and $\yy = (0,1)$, with each letter of $\WW$ being chosen independently and uniformly at random. For any (spanning) subgraph $G \subset \Z^2$ of the lattice, the random walk $\WW$ then induces a (random) walk $\WW_G$ on $G$: starting at the origin, we consider the letters of $\WW$ one at a time, and for each letter of $\WW$, we take a step in the appropriate direction in $G$ provided the edge in question is present in $G$, and stand still otherwise. For any \emph{fixed} $G \subset \Z^2$, it is clear that $\WW_G$ is distributed as the simple random walk on $G$, and since $G$ is a subgraph of a recurrent graph, we conclude that $\WW_G$ is almost surely recurrent. It follows immediately from Fubini's theorem that the random walk $\WW$ almost surely has the property that the induced walks $\WW_G$ are recurrent for \emph{almost all} $G \subset \Z^2$. We are then naturally led to the following question: does the random walk $\WW$ almost surely have the property that the induced walks $\WW_G$ are recurrent for \emph{all} $G \subset \Z^2$? Our first result answers this question negatively in the following strong sense.

\begin{theorem}\label{t:main}
	If $\WW$ is a simple random walk on $\Z^2$, then there almost surely exists a (random) exceptional subgraph $H \subset \Z^2$ for which the induced walk $\WW_H$ 
	\begin{enumerate} 
		\item visits each site reachable from the origin in $H$ finitely many times, and
		\item fails to visit infinitely many sites reachable from the origin in $H$.
	\end{enumerate}
\end{theorem}
More generally, we can ask whether a countably infinite independent collection of simple random walks almost surely has the property that, for every $G\subset\Z^2$, at least one of the walks in this collection induces a recurrent walk on $G$. An extension of the proof of Theorem~\ref{t:main} allows us to prove the following result, which answers this question negatively as well.
\begin{theorem}\label{t:countable}
	If $\{\WW_i\}_{i \in \mathbb{N}}$ is a collection of independent simple random walks on $\Z^2$, then there almost surely exists a (random) exceptional subgraph $H \subset \Z^2$ so that, for every $i\in \mathbb{N}$, the induced walk $(\WW_i)_H$ 
	\begin{enumerate} 
		\item visits each site reachable from the origin in $H$ finitely many times, and
		\item fails to visit infinitely many sites reachable from the origin in $H$.
	\end{enumerate}
\end{theorem}

What about an uncountable collection of simple random walks? To avoid measurability issues, we need to be careful about how we phrase such a question. One natural formulation is in the language of branching random walks. A \emph{branching random walk} on $\Z^d$ starts with a single particle at the origin. At each time step, each particle independently generates a number of additional particles at its current location according to some fixed \emph{offspring distribution}, and we say that this offspring distribution is \emph{nontrivial} if the number of offspring is nonzero with positive probability. Independently of the other particles and their history, all particles then take a step in a direction chosen uniformly at random, which leads to a random family of dependent simple random walks. In this language, one can then ask whether a branching random walk on $\Z^2$ almost surely has the property that, for every $G\subset\Z^2$, at least one of the branches of the branching random walk induces a recurrent walk on $G$. Our final result answers this question positively, and furthermore, shows that the same is true in dimensions greater than two as well.
\begin{theorem}\label{t:branching}
Fix  $d \in \N$ and let $(\WW_i)_{i \in S}$ be a family of random walks generated by a branching random walk on $\Z^d$ with a nontrivial offspring distribution. Then, almost surely, for every (spanning) subgraph $G \subset \Z^d$, there is some $j \in S$ for which the induced walk $(\WW_j)_G$ is recurrent.
\end{theorem}

Our work here fits into the broader context of attempting to understand the robustness of objects such as random walks and Brownian paths, in terms of their quasi-everywhere properties or their dynamical sensitivity; see~\citep{Fuk, Kono, dyn} for example. 

Let us mention two results in this general direction that are particularly close to our results in spirit. The first relevant result is a theorem of Adelman, Burdzy and Pemantle~\citep{miss} on the projections of three-dimensional Brownian motion. The projection of Brownian motion in $\R^3$ onto any \emph{fixed} plane yields Brownian motion in that plane which is neighbourhood recurrent; Adelman, Burdzy and Pemantle~\citep{miss} however show that there almost surely exists a (random) exceptional plane on which the projection is not neighbourhood recurrent. The second result that is pertinent is a theorem of Hoffman~\citep{sens} demonstrating that recurrence of the simple random walk on $\Z^2$ is dynamically sensitive; in other words, if the coin tosses of the random walk are refreshed continuously with Poisson clocks generating a \emph{dynamic random walk}, then although dynamic random walk is almost surely recurrent at any \emph{fixed} time, there almost surely exists a (random) exceptional time at which the dynamic random walk is not recurrent. 

Although the two results mentioned above bear some similarities in flavour to our first two results, it is perhaps worth remarking that the methods of proof are somewhat different: while the results in both~\citep{miss} and~\citep{sens} are based on second moment computations, the proofs of Theorems~\ref{t:main} and~\ref{t:countable}, in contrast, proceed by explicitly `embedding drift'.

This note is organised as follows. In Section~\ref{s:proof}, we first sketch a natural approach that fails, but nonetheless motivates our construction, then we prove Theorem~\ref{t:main}, and we finally sketch how the same argument extends to prove Theorem~\ref{t:countable}. Section~\ref{s:branching} is devoted to the proof of Theorem~\ref{t:branching}. We conclude in Section~\ref{s:conc} with a discussion of some open problems.

\section{Existence of exceptional subgraphs}\label{s:proof} 
In this section, we prove Theorems~\ref{t:main} and~\ref{t:countable}. Before we do so, let us sketch a construction which, while failing to prove Theorem~\ref{t:main}, serves as the motivation for the construction in our proof. 

Suppose that $\WW$ is a simple random walk on $\Z^2$. Let us construct a (random) subgraph $P \subset \Z^2$ which exhibits some drift. We shall ensure that $P$ is an infinite non-decreasing path (i.e., a north-east path) passing through the origin, and we reveal $P$ as follows. We shall read off the letters of $\WW$ one at a time and follow the induced walk $\WW_P$ on $P$, revealing more of $P$ as and when $\WW_P$ needs to know if a particular edge is present in $P$. At any finite time, it is clear that $P$ (or rather, what has been revealed of $P$ so far) is a finite non-decreasing path through the origin, and we are forced to reveal more of $P$ at this time if and only if $\WW_P$ is at one of the leaves of $P$ and the next letter of $\WW$ would cause $\WW_P$ to exit $P$ in a non-decreasing fashion. Our strategy for constructing $P$ is then as follows: if $\WW_P$ is at the north-eastern leaf of $P$ at some stage, and the next letter of $\WW$ causes $\WW_P$ to travel either north or east, we extend $P$ so as to allow this, proceeding analogously at the south-western leaf as well. 

What can we say about the induced walk $\WW_P$ on the path $P$ as constructed above? It is not hard to see that if we identify $P$ with $\Z$ by `unrolling' it, then $\WW_P$ is a random walk on $\Z$ with the following law: if $[a,b]\subset \Z$ is the range of the walk at some time and the walk is at $x \in [a,b]$ at this time, then the walk moves to either $x-1$ or $x+1$ both with probability $1/2$, unless $x \in\{ a, b\}$, in which case, the walk moves to $b+1$ with probability $2/3$ and to $b-1$ with probability $1/3$ at $x = b$, and similarly to $a-1$ with probability $2/3$ and to $a+1$ with probability $1/3$ at $x = a$. In other words, for the path $P$ constructed as described above, the induced walk $\WW_P$ behaves like a random walk on $\Z$ with a tiny amount of drift; indeed, the walk possesses some drift away from the origin when it is at the boundary of its range, but behaves like the simple random walk in the interior of its range. Unfortunately, this tiny amount of drift does not stop $\WW_P$ from being recurrent, but this construction nevertheless demonstrates that it is possible to construct (random) subgraphs of $\Z^2$ where the induced walk possesses some drift; below, we prove Theorem~\ref{t:main} with a more careful construction that embeds more drift into the induced walk.

We need two simple facts about the simple random walk on $\Z$. First, we require the following well-known fact.
\begin{proposition}\label{p:int}
	The probability that the simple random walk on the interval $\{0,1,\dots,n\}$ started at $1$ visits $n$ before it visits $0$ is $1/n$.\qed
\end{proposition}
Next, we shall also make use of the following crude bound.
\begin{proposition}\label{p:exp}
	The expected number of times the simple random walk on $\Z$ started at $0$ visits $0$ in the first $N$ steps is at most $10\sqrt{N}$.\qed
\end{proposition}

Armed with these two facts, we are now ready to prove our main result.
\begin{figure}
	\begin{center}
		\trimbox{0cm 0cm 0cm 0cm}{ 
			\begin{tikzpicture}[xscale = 1.2, yscale = 0.8]
			
			\draw[very thick] (1, 0) -- (1, 7) [dashed];
			\draw[very thick] (8, 0) -- (8, 7) [dashed];
			\draw[very thick] (1, 1) -- (1, 6);
			\draw[very thick] (8, 1) -- (8, 6);
			
			\draw[] (2, 1) -- (2, 6)[red];
			\draw[] (3, 1) -- (3, 6)[red];
			\draw[] (4, 1) -- (4, 6)[red];
			\draw[] (5, 1) -- (5, 6)[red];
			\draw[] (6, 1) -- (6, 6)[red];
			\draw[] (7, 1) -- (7, 6)[red];
			
			\draw[very thick] (1, 1) -- (4, 1);
			\draw[very thick] (1, 2) -- (5, 2);
			\draw[very thick] (1, 3) -- (3, 3);
			\draw[very thick] (1, 4) -- (8, 4);
			\draw[very thick] (1, 5) -- (2, 5);
			\draw[very thick] (1, 6) -- (7, 6);

			\draw[] (8, 1) -- (4, 1)[red];
			\draw[] (8, 2) -- (5, 2)[red];
			\draw[] (8, 3) -- (3, 3)[red];
			\draw[] (8, 5) -- (2, 5)[red];
			\draw[] (8, 6) -- (7, 6)[red];
			
			\foreach \x in {1,2,3,4,5,6,7,8}
			\foreach \y in {1,2,3,4,5,6}
			\node (\x\y) at (\x, \y) [inner sep=0.5mm, circle, fill=black!100] {};
			\end{tikzpicture}
		}
	\end{center}
	\caption{A slice of $H$ between two consecutive vertical lines in $H$; edges coloured black are present while edges coloured red are absent.}\label{pic:sketch}
\end{figure}
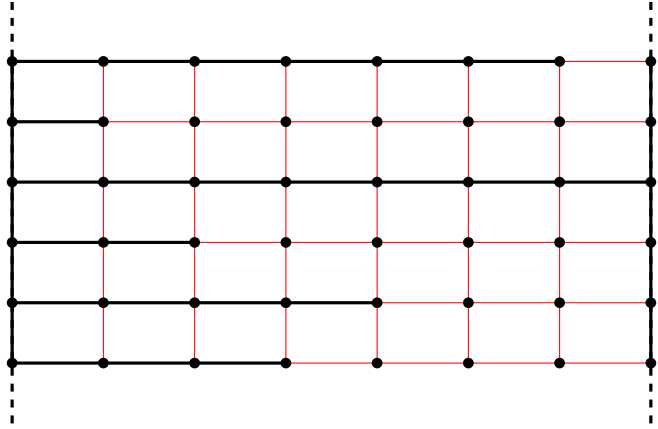

\begin{proof}[Proof of Theorem~\ref{t:main}]
	To prove the theorem, we will construct a (random) graph $H$ based on the random walk $\WW$ where the induced walk $\WW_H$ exhibits a strong drift away from the origin.
	
	The graph $H$ we construct will consist of the vertical lines $L_n = \{(x,y)\in\Z^2 : x=2^n-1\}$ for all integers $n\ge 0$, and a (random) collection of finite horizontal segments between any consecutive pair of vertical lines with the property that exactly one such horizontal segment connects any consecutive pair of vertical lines; here, by lines and segments, we mean the edges in the appropriate paths in $\Z^2$, as shown in Figure~\ref{pic:sketch}.
	
	Let us denote the location of the induced walk $\WW_H$ at a given time $t\ge0$ by $\WW_H(t) = (X_H(t), Y_H(t))$, with the time $t$ tracking steps along $\WW_H$ (as opposed to $\WW$). As before, we shall read off the letters of $\WW$ one at a time, and we shall reveal $H$ by following the induced walk $\WW_H$ and revealing more of $H$ as necessary. 
	
	Notice that the only vertical edges in $H$ are (deterministically) those on a vertical line $L_n$ for some $n \ge 0$, so at any time $t\ge0$, the induced walk $\WW_H$ accepts a vertical step of $\WW$ if and only if $\WW_H(t) \in L_n$ for some $n\ge0$. 
	
	We reveal the horizontal edges of $H$ in stages: during stage $n \ge 0$, we shall reveal all the horizontal edges of $H$ in between the lines $L_n$ and $L_{n+1}$, with the stage ending as soon as there is a horizontal path connecting these vertical lines in $H$. Note in particular that during stage $n$, we have already revealed all the horizontal edges in $H$ between $L_0$ and $L_{n}$, and none of the horizontal edges in $H$ to the right of $L_{n+1}$. 
	
	We begin by declaring every horizontal edge to the left of $L_0$ as being absent in $H$, and for $n \ge 0$, having completed stage $n-1$, we reveal $H$ in the following fashion. At some $t \ge 0$ during stage $n$, there are two possibilities. If $X_H(t) < 2^{n}-1$, then we have nothing to do when we read off the next letter of $\WW$ since all the edges of $H$ to the left of $L_n$ have already been revealed. If $X_H(t) \ge 2^{n} - 1$ on the other hand, we reveal $H$ in such a way so as to ensure that $\WW_H$ always accepts a letter of $\WW$ that would cause the induced walk to travel to the right. The stage ends as soon as we have a single horizontal path connecting $L_n$ and $L_{n+1}$, or in other words, at the first time $t \ge 0$ when we have $X_H(t) = 2^{n+1} - 1$. At the end of the stage, all horizontal edges between $L_n$ and $L_{n+1}$ whose presence or absence in $H$ have not been revealed over the course of the stage, we declare as being absent in $H$. In particular, the line $L_{n+1}$ is incident to precisely one horizontal edge to its left in $H$ and it has been revealed by the end of the stage.
	
	The above construction clearly ensures that $H$ has the structure we promised. More is true, however; as we shall shortly see, our construction endows the induced walk $\WW_H$ with a strong drift to the right. 
	
	For $n\ge0$, let $\tau_n$ be the first time $t$ at which we have $X_H(t) = 2^n-1$, and let $E_n$ denote the event that the walk $(\WW_H(t))_{t \ge \tau_n}$ hits the line $L_{n-1}$ before hitting the line $L_{n+1}$; in other words, $E_n$ is the event that there exists a time $t \in [\tau_n, \tau_{n+1})$ at which $X_H(t) = 2^{n-1} - 1$. With these definitions in place, we have the following claim.
	
	\begin{claim}\label{claim:p}
		There exists an absolute constant $c \in (0,1)$ such that $\P(E_n) \le c^n$ for all $n \ge 0$.
	\end{claim}
	\begin{proof}
		Let $\alpha$ be the $y$-coordinate of the unique horizontal path joining the vertical lines $L_{n-1}$ and $L_n$, and let $T$ be the first time after $\tau_n$ at which the walk $\WW_H$ hits either the line $L_{n-1}$ or the line $L_{n+1}$. Note that the time between $\tau_n$ and $T$ naturally decomposes into \emph{excursions}, where an excursion is a maximal interval of time during which the walk $\WW_H$ remains at some fixed $y$-coordinate.
		
		Let us now describe the walk $\WW_H$ in terms of its excursions. First, note that our construction of $H$ ensures that the $y$-coordinates of successive excursions of $\WW_H$ are determined by a simple random walk on $\Z$ started at $\alpha$. Also, we can describe an excursion at some $y$-coordinate $\beta$ as follows. If $\beta \ne \alpha$, then during an excursion at $\beta$, successive $x$-coordinates of $\WW_H$ are determined by a simple random walk on the interval $\{2^n-1,2^n, \dots, 2^{n+1}-1\}$ started at $2^n-1$, with the excursion ending either, with probability 1, when the $x$-coordinate of the walk is $2^{n+1}-1$, or, with probability $2/3$, when the $x$-coordinate of the walk is $2^{n}-1$. If $\beta = \alpha$ on the other hand, then successive $x$-coordinates of $\WW_H$ are determined by a simple random walk on the interval $\{2^{n-1}-1,2^{n-1}, \dots, 2^{n+1}-1\}$ started at $2^n-1$, with the excursion ending either, with probability 1, when the $x$-coordinate of the walk is either $2^{n-1}-1$ or $2^{n+1}-1$, or, with probability $1/2$, when the $x$-coordinate of the walk is $2^{n}-1$. Crucially, note that by the strong Markov property of the walk $\WW_H$, each excursion depends on past excursions only through the endpoint of the last excursion preceding it.
		
		Let us say that an excursion is \emph{positively successful} if it ends on account of the walk $\WW_H$ reaching the line $L_{n+1}$, and \emph{negatively successful} if it ends on account of the walk $\WW_H$ reaching the line $L_{n-1}$. In this language, we see that $E_n$ is precisely the event that we witness a negatively successful excursion before a positively successful one.
		
		We would like to show that in the first $3^n$ excursions, there is a positively successful excursion, but no negatively successful one. A minor technicality arises from the fact that excursions are only defined until the first successful one occurs. To circumvent this issue, consider a modified process coupled with $\WW_H$ which, after a successful excursion with $y$-coordinate $\beta$, teleports to the line $L_n$, taking $y$-coordinate $\beta+1$ with probability $1/2$, and $\beta-1$ otherwise, before again moving according to $\WW_H$. Let $F_1$ to be the event that at least one of the first $3^n$ excursions in this modified process is positively successful, and let $F_2$ be the event that none of the first $3^n$ excursions in the modified process are negatively successful. Since $\WW_H$ and our modified process behave identically until the first successful excursion, if $F_1$ and $F_2$ both occur, then $E_n$ must occur. Therefore, it suffices to show that both $F_1$ and $F_2$ are overwhelmingly likely.
		
		First, we deal with the event $F_1$. It is easy to see from Proposition~\ref{p:int} that an excursion is positively successful with probability at least $(1/100)2^{-n}$. Using the strong Markov property, we conclude that  
		\[\P(F_1^c) \le \left(1-2^{-n}/100\right)^{3^n} \le c_1^n, \]
		where $c_1 \in (0,1)$ is an absolute constant.
		
		Next, we handle the event $F_2$. Notice that we may only witness a negatively successful excursion at $y$-coordinate $\alpha$; with this in mind, let $Z$ be the number of excursions in the first $3^n$ excursions at $y$-coordinate $\alpha$. Since the $y$-coordinates of successive excursions are determined by a simple random walk on $\Z$ started at $\alpha$, we conclude from Proposition~\ref{p:exp} that 
		\[\P\left(Z \ge (7/4)^n\right) \le \frac{10 \sqrt{ 3^{n}}}{ (7/4)^n} \le c_2^n,\]
		where $c_2 \in (0,1)$ is an absolute constant. As before, we know from Proposition~\ref{p:int} that an excursion at $y$-coordinate $\alpha$ is negatively successful with probability at most $100 \cdot 2^{-n}$, so we may again use the strong Markov property to conclude that
		\begin{align*}
		\P(F_2 ^c) &\le \P\left(Z \ge (7/4)^n\right) + \P\left(F_2^c \Cond Z < (7/4)^n\right)\\
		&\le c_2^n + \left(1 - (1-100 \cdot 2^{-n})^{(7/4)^n}\right)\\
		&\le c_2^n + c_3^n,
		\end{align*}
		where $c_3 \in (0,1)$ is an absolute constant.
		
		The result follows from the estimates above since $\P(E_n) \le \P(F_1 ^c) + \P(F_2^c)$.
	\end{proof}
	It follows from the above claim, by the Borel--Cantelli lemma, that the walk $\WW_H$ almost surely visits the line $L_n$ only finitely many times for each $n \ge 0$, thus proving the result.
\end{proof}

The proof of Theorem~\ref{t:countable} follows from a simple modification of the proof of Theorem~\ref{t:main}; therefore, we only provide a sketch highlighting the main differences. 

\begin{proof}[Proof of Theorem~\ref{t:countable}]
	As in the argument above, we begin with the vertical lines $L_1,L_2,\dots$, where $L_n$ has $x$-coordinate $2^n - 1$. Suppose that we have defined the exceptional graph $H$ up to $L_{i}$ by using the walks $\WW_1, \dots, \WW_i$. To define the portion of $H$ between $L_i$ and $L_{i+1}$, we run $\WW_{i+1}$ on the already defined portion of $H$ until it first hits $L_i$. At that point, we run the algorithm used in the proof above simultaneously on the $i+1$ walks $\WW_1, \dots, \WW_{i+1}$ (so that they can move freely to the right of $L_i$ but not up or down in that region), stopping each walk when it first hits $L_{i+1}$.  We continue this process, introducing one new walk at each step. The analysis goes through essentially as before: if $E_{n,i}$ is the event that, after hitting $L_n$, the walk $\WW_i$ hits $L_{n-1}$ before $L_{n+1}$, then as in Claim~\ref{claim:p}, we have $\P(E_{n,i})\le nc^n$ for all $n\ge i$ for some absolute constant $c\in (0,1)$. The result again follows from the Borel--Cantelli lemma.
\end{proof}

\section{Non-existence of exceptional subgraphs} \label{s:branching}
We now consider branching random walks on $\Z^d$ in all dimensions $d \ge 2$. Recall that we start with a single particle at the origin, and that at each time step, every particle independently generates a number of additional particles at its current location according to some fixed nontrivial distribution, where at least one new particle is generated with positive probability; independently of the other particles and their history, all particles then take a step in a direction chosen uniformly at random. The result is a random family $(\WW_i)_{i\in S}$ of dependent simple random walks, where the branches $\WW_i$ with $i \in S$ each correspond to a walk on $\Z^d$ obtained by starting with the original particle at the origin and then following either the particle presently under consideration or one of its children at each time.

We will use the following Chernoff-type bound; see~\citep{Chernoffcite} for a proof.
\begin{proposition}\label{chernn}
For $n \in \mathbb{N}$, $p \in (0,1)$ and $\eps>0$, we have
\[
\P(\Bin(n,p) \le np(1-\eps)) \le \exp\left(-\frac{\eps^2 np}{2}\right). \eqno\qed
\]
\end{proposition}
We are now ready to give the proof of our final result.
\begin{proof}[Proof of Theorem~\ref{t:branching}]
Note that removing particles reduces the set of walks and makes the problem of finding a recurrent walk harder. Hence, if at some point a particle has more than one child, we may discard all but one of its children. We may therefore assume that, for some $\eps > 0$, each particle either has one child with probability $\eps$, or no children at all with probability $1 - \eps$. It will be helpful to have some notation: in what follows, we write $B_G(x,\rho)$ to denote the set of vertices at graph distance at most $\rho$ from $x$ in a graph $G$, write $B(x,\rho)$ for $B_{\Z^d}(x,\rho)$, and abbreviate $B(0,\rho)$ by $B(\rho)$.
	
In this proof, we will need to show that, with very high probability, certain particles have exponentially many children at a given future time and ensure that most of these children do not wander very far from the origin. To this end, we need two results which we state and prove below.

First, we need the following estimate for the probability of a simple random walk on $\Z^d$ getting unexpectedly far from its starting point after taking some finite number of steps.	
\begin{claim}\label{cl:bound}
Fix $\delta > 0$ and let $S_n$ be an $n$-step simple random walk on some subgraph $G \subset \Z^d$ starting from the origin. Then we have
\[ \P \left( S_n \not\in B_G \left( 0 , \delta n \right)\right) \le c_d n^d \exp \left( - \frac{\delta^2 n}{2} \right),\]
where $c_d > 0$ is a constant depending only on the dimension $d$.
\end{claim}
\begin{proof}
We make use of an old bound on the transition probabilities of a Markov chain due to Varopoulos~\citep{V} and Carne~\citep{C}, although we use it in a more recent form due to Peyre~\citep{VC}. To state this bound, we need a little set up. Given a graph $G$ and a pair of vertices of $x$ and $y$ of $G$ in the same connected component, let $p_t(x,y)$ denote the probability that a simple random walk on $G$ starting at $x$ is at $y$ after $t$ steps, and write $\rho(x,y)$ for the graph distance between $x$ and $y$ in $G$; in this language, we have
\[p_t(x,y) \le 2 \sqrt{\frac{  \deg(y) }{\deg(x)}} \exp \left( - \frac{ \rho(x,y) ^2}{2t} \right).\]

With the above bound in hand, we conclude that the probability that an $n$-step simple random walk on some subgraph $G \subset \Z^d$ starting at the origin ends up outside the ball $B_G(0, \delta n)$ is at most
\[
\sum_{y \in B(n) \setminus B_G(0, \delta n) } p_n(0, y) \le \sum_{y \in B(n)} 2 \sqrt{2d} \exp \left( - \frac{\delta^2 n^2}{2n} \right) \le \sqrt{8d} (2n + 1)^d \exp \left( - \frac{\delta^2 n}{2} \right);
\]
it is clear that the bound above is of the required form, proving the claim.
\end{proof}
	
Second, we need the following estimate for the rate of growth of a Galton--Watson branching process.
\begin{claim}\label{cl:descendants}
Let $(N_j)_{j \ge 0}$ be the number of descendants at time $j$ of a Galton--Watson branching process started with a single particle and with an offspring distribution that takes the value $2$ with probability $\eps$ and $1$ with probability $1 - \eps$. Then there exists constants $c, c' >0$ such that
\[\P \left( N_j \ge (1+c)^{j} \right) \ge 1 - e^{-c' j}\]
for all $j\ge 1$.
\end{claim}
	
\begin{proof}
Conditioned on $N_{j-1}$, the random variable $N_j$ is distributed as $\Bin ( N_{j-1}, \eps ) + N_{j-1}$, independent of everything else in the past. For $j\ge 1$, define $Y_{j} = N_{j}/N_{j-1}$ and note that $Y_j \in [1,2]$ and $\E[ Y_j ] = 1 + \eps$. From Proposition~\ref{chernn}, we deduce that
\begin{align*}
\P \left( Y_j \ge 1 + \eps/2 \right) &= \sum_{i=1}^\infty \P \left( \Bin \left( i, \eps \right) \ge i\eps/2  \Cond  N_{j-1} = i \right) \P \left(N_{j-1} = i \right)\\
		&\ge \sum_{i=1}^\infty \left( 1  - e^{- i\eps/8} \right)\P \left(N_{j-1} = i \right)\\
		&\ge \left( 1  - e^{- \eps/8} \right).
\end{align*}
Thus, setting $p = ( 1  - e^{- \eps/8})$, there exist independent random variables $(X_j)_{j\ge 1}$ dominated by the $Y_j$ such that $X_j$ takes the value $1+ \eps/2$ with probability $p$ and $1$ otherwise with probability $1-p$. Since $N_j=Y_jY_{j-1}\cdots Y_1\ge X_j X_{j-1}\cdots X_1$, by Proposition~\ref{chernn}, we have
\begin{align*}
		\P \left( N_j \ge (1 + \eps/2)^{jp/2} \right) &\ge \P \left( \bigl| \left\{i\in [j] : X_i >1 \right\}\bigr| \ge \frac{jp}{2} \right)\\
		&\ge \mathbb{P} \left( \Bin\left(j,p \right) \ge \frac{jp}{2} \right)\\
		&\ge 1 - e^{-\frac{jp}{8}},
		\end{align*}
as required.
\end{proof}
	
Before diving into the details of the proof, let us sketch our plan of attack. Our argument will proceed in \emph{stages} bookended	 by a rapidly-increasing sequence of times $\left(T_i \right)_{i \in \N}$.  At each time $T_i$, and for every possible finite subgraph $G \subset \Z^d$ contained in the box $[-T_i,T_i]^d$, we choose a representative particle $p_G$. We then show that, with very high probability, for each such representative particle $p_G$ and every possible extension of the corresponding graph $G$ onto $[-T_{i+1},T_{i+1}]^d$, some descendant of $p_G$ visits every vertex of $G$ with distance at most $i$ (in $G$) to the origin, and ends up exactly at or one step away from the origin at time $T_{i+1}$. We shall show that the failure probabilities decay rapidly enough so that we may finish by applying the Borel--Cantelli lemma. To avoid clutter, we will not worry about making sure all the appropriate values are integers.
	
Let $T_0 = 0$, and for the singleton graph containing just the origin, we choose the initial particle at the origin. Let $\delta$ be sufficiently small such that $2^{-4d\delta}(1+c)^{(1-4\delta)} > 1$ and $c \delta - c' (1-5 \delta) < 0$, where $c, c'$ are the absolute constants promised by Claim~\ref{cl:descendants}.
	
Suppose now that we have run the branching walk over the course of $\ell - 1$ stages until time $n = T_{\ell-1}$, and that we have a representative particle $p_G$ at some location in $[-n, n]^d$ for each of the at most $2^{d(2n+1)^d}$ possible finite subgraphs $G  \subset \Z^d$ contained in the box $[-n,n]^d$. Note that these particles need not be distinct, and that a single particle may be at different locations in different subgraphs. We now describe how we construct the $\ell$th stage of the branching walk.
	
Let $N>n/\delta$ be a suitably large integer. We specify what conditions $N$ needs to satisfy in what follows, and shall then take $T_\ell = N$. We shall further divide the $\ell$th stage, which consists of the interval of time $[n+1, N]$, into three smaller blocks of time as follows.

In the first block, we run the branching walk for $\delta N$ more steps (after time $n$) so that for each $G \subset \Z^d$ contained in the box $[-n,n]^d$, the representative particle $p_G$ has a set of descendants $P_G$ after these $\delta N$ steps. By Claim~\ref{cl:descendants}, any fixed representative particle $p_G$ has at least $(1+c)^{\delta N}$ descendants with probability at least $1 - e^{-c' \delta N}$. If even one of the representative particles has not branched this much, we declare that this step has \emph{failed}. Note that we are free to discard particles, so we may assume (as long as this step has not failed) that each representative particle $p_G$ has a set $P_G$ of exactly $(1+c)^{\delta N}$ descendants. Note that all the particles under consideration are at a graph distance of at most $n + \delta N<2\delta N$ from the origin in all the graphs under consideration. 
	
In the second block, we run the branching walk for another $N - 5 \delta N$ steps. For each $G \subset \Z^d$ contained in the box $[-n,n]^d$, we now count the number of descendants of each $q \in P_G$. We assume that every such $q\in P_G$ has at least $(1+c)^{N - 5 \delta N}$ descendants, and if this ever fails to hold, we again declare that the step has \emph{failed}. By Claim $\ref{cl:descendants}$, the probability of failure for any particular particle $q$ as above is at most $\exp \left( - c' (N - 5\delta N) \right)$.
	
Now, fix $H \subset \Z^d$ contained in the box $[-N, N]^d$, and suppose that it induces a graph $H'$ on $[-n, n]^d$. We say that a particle $q \in P_{H'}$ is \emph{$H$-good} if at least half of its descendants are no further than $3 \delta N$ from the origin in $H$ after the $n + N - 4 \delta N$ steps taken so far. Let us estimate the probability  $\lambda_H$ that any given particle $q \in P_{H'}$ is $H$-good. Fix $q\in P_{H'}$  and denote the position of $q$ at time $n + \delta N$ by $L_q$. By Claim $\ref{cl:bound}$, the probability that a given descendant of $q$ is within distance $\delta N$ of $L_q$ in $H$, and hence within distance $3 \delta N$ of the origin, is at least 
\[1 - c_d N^d \exp \left( - \frac{\delta^2 N}{2} \right).\] 
Let $Y$ be the proportion of descendants of $q$ which are within $3\delta N$ of the origin. By the linearity of expectation, 
\[\E \left[ Y \right] \ge 1 - c_d N^d \exp \left( - \frac{\delta^2 N}{2} \right).\]
If $q$ is not good, then $Y\le 1/2$, and since $Y\le 1$, we find that $\lambda_H + \tfrac 12 (1-\lambda_H) \ge \mathbb{E} \left[ Y \right]$. Hence, \[\lambda_H \ge 1 - 2 c_d N^d \exp \left( - \frac{\delta^2 N}{2} \right)\] and, by choosing $N$ sufficiently large, we may certainly assume that $\lambda_H \ge 2/3$. We now also declare the step to have \emph{failed} if, for some $H$ as above, at most $\frac{1}{3}(1+c)^{\delta N}$ elements of $P_{H'}$ are $H$-good, and deduce from Proposition~\ref{chernn} that the probability of failing in this fashion for any fixed $H$ is at most
	\begin{align*}
	\mathbb{P} \left( \bin \left( (1+c)^{\delta N}, \lambda \right) \le \tfrac{1}{3} \left(1+c\right)^{\delta N} \right) & \le \mathbb{P} \left( \bin \left( (1+c)^{\delta N}, \lambda \right) \le \tfrac{1}{2} \left(1+c\right)^{\delta N} \lambda \right) \\
	&\le \exp \left( - \frac{1}{8} (1+c)^{\delta N} \lambda \right)\\
	&\le \exp \left( - \frac{1}{16} (1+c)^{\delta N} \right).
	\end{align*}
Hence, if we have not already failed, then for any $H$ as above, counting the descendants of $H$-good particles that don't stray too far from the origin, we have at least $\tfrac{1}{6} \left( 1 + c \right)^{N - 4 \delta N}$ such descendants at graph distance at most $3 \delta N$ from the origin in $H$ after $n + N - 4 \delta N$ steps; we call these descendants $H$-\emph{counters}.
	
In the third and final block of the stage, our goal is to visit all vertices at distance at most $\ell - 1$ from the origin in every $H \subset \Z^d$ contained in the box $[-N, N]^d$. For any fixed $H$ as above, there are at most $(2 \ell - 1)^d$ such vertices close to the origin. We enumerate these vertices and pick a path of length at most $\ell-1$ from the origin to each such vertex in $H$; putting these together, we get a walk $W_H$ of length at most $(2\ell - 2)(2\ell-1)^d$ in $H$ which starts at the origin, ends at the origin, and visits every vertex at most distance $\ell - 1$ from the origin; furthermore, by choosing $N$ sufficiently large, we may assume that $ (2\ell - 2)(2\ell-1)^d \le \delta N - n$.

For each $H$-counter $v$, let $A_v$ be the event that, in the next $4 \delta N -n$ steps, $v$ visits every vertex at most distance $\ell - 1$ from the origin. If no $A_v$ occurs, for any $H$ as above, we again declare the stage to have \emph{failed}. Conditional on the positions of the $H$-counters at the start of this block, the events $A_v$ are independent of one another. Now, fix these `starting positions' of the $\tfrac{1}{6}(1+c)^{N-4\delta N}$ $H$-counters. For each $H$-counter $v$, there is a path of length at most $3\delta N$ from its starting position to the origin. Let $B_v$ be the event that $v$ strictly follows this path, then the walk $W_H$ and then never leaves $B_H(0,1)$. Clearly $B_v$ implies $A_v$, and so $\P[A_v] \ge \P[B_v] \ge (2d)^{-4\delta N}$. Hence, conditional on the starting positions of the $H$-counters, the probability that no $A_v$ occurs is at most
	\[(1 - (2d)^{- 4\delta N})^{\tfrac{1}{6} (1+c)^{N - 4 \delta N}} \le \exp \left( - \frac{1}{6} (2d)^{-4\delta N}  (1+c)^{N - 4 \delta N}\right).\]
Since the bound above is independent of the starting positions of the $H$-counters, it also holds without conditioning for any fixed $H$ as above.

In summary, we declare the $\ell$th stage to have failed if one of the following happens.
\begin{enumerate}
	\item In the first block, there is some $G$ contained in $[-n,n]^d$ whose representative particle $p_G$ does not branch sufficiently. This happens with probability at most 
	\begin{equation}\label{f1}2^{d(2n+1)^d} \exp({- c'\delta N}).\end{equation}
	\item In the second block, there is some $G$ contained in $[-n,n]^d$ for which some $q \in P_G$ does not branch sufficiently. This happens with probability at most 
	\begin{equation}\label{f2}2^{d(2n+1)^d} (1+c)^{\delta N} \exp \left( - c' (N - 5\delta N) \right).\end{equation}
	\item In the second block, there is some $H$ contained in $[-N,N]^d$ for which too few particles in the corresponding set $P_{H'}$ are $H$-good. This happens with probability at most 
	\begin{equation}\label{f3}2^{d(2N+1)^d} \exp \left( - \frac{1}{16} (1+c)^{\delta N} \right).\end{equation}
	\item In the third block, there is some $H$ contained in $[-N,N]^d$ for which no $H$-counter visits every vertex at graph distance most $\ell-1$ from the origin during steps $n+N-4\delta N$ through $N$ in $H$. This happens with probability at most 
	\begin{equation}\label{f4}2^{d(2N+1)^d} \exp \left( - \frac{1}{6} (2d)^{-4\delta N}  (1+c)^{N - 4 \delta N}\right). \end{equation}
\end{enumerate}

We now finish the proof as follows. By the union bound, we conclude that the probability of the $\ell$th stage being declared a failure is at most the sum of the estimates in~\eqref{f1},~\eqref{f2},~\eqref{f3} and~\eqref{f4}; by the choice of $\delta$, for fixed $n$, this sum tends to 0 as $N \to \infty$. We choose $N$ large enough to both ensure that the bounds above hold and to make the probability of the $\ell$th stage being declared a failure at most $2^{-\ell}$, and as mentioned earlier, we set $T_{\ell} = N$. In the case of success, for each $H \subset \Z^d$ contained in the box $[-T_{\ell}, T_{\ell}]^d$, we pick an arbitrary particle that walked the chosen path above and take that particle to be $p_H$, whereas in the case of failure, we choose representative particles arbitrarily. The Borel--Cantelli lemma now implies that almost surely, there are only finitely many stages that fail, so for every subgraph, there is a particle that visits every reachable vertex infinitely often.
\end{proof}
\section{Conclusion}\label{s:conc}
We have shown that a countable collection of independent simple random walks in two dimensions can almost surely be made transient by dropping to a suitable random two-dimensional subgraph; on the other hand, in any number of dimensions, a branching random walk is almost surely recurrent on every subgraph, in the sense that some branch is recurrent on each subgraph. Natural intermediate questions arise from considering the dynamic random walk $\WW_d^{(t)}$ mentioned earlier. This object was introduced by Benjamini, H\"aggstr\"om, Peres and Steif~\citep{dyn}, who showed that in three or four dimensions (i.e., when $d \in \{3,4\}$), there is almost surely some time $T$ such that $\WW_d^{(T)}$ is recurrent; while, in five or more dimensions, almost surely the walk is transitive at all times. In two dimensions, Hoffman~\citep{sens} showed that there is almost surely a time $T$ such that the walk is $\WW_2^{(T)}$ transitive; see also~\citep{AH}. In the light of these facts, the following question seems of interest.
\begin{problem}
Let $d\in\{2,3,4\}$, and let $\WW_d^{(t)}$ be a dynamic random walk on $\Z^d$. Is there almost surely a (random) subgraph $H \subset \Z^d$ on which $\WW_d^{(t)}$ is transitive for every time $t \ge 0$? 
\end{problem}
It seems plausible that the answer is positive in four dimensions and negative in two dimensions; we are not prepared to offer a guess in three dimensions, however. 

Another interesting question concerns paths. It is clear that for any path $P$ in ${\mathbb Z}^d$ through the origin, a random walk on $P$ is almost surely recurrent. In Section~\ref{s:proof}, we noted that when one attempts to build a (random) path in two dimensions that greedily forces the random walk north and east, the resulting induced walk is almost surely recurrent; the induced walk exhibits some drift (compared to a fixed path) but only at the end points, and this is not enough to make it transient. This suggests the following question.
\begin{problem}
Let $\WW$ be a simple random walk on $\Z^d$.  Is it almost surely the case that $\WW_P$ is recurrent for every path $P \subset \Z^d$ through the origin?
\end{problem}

In the cases where we can force a simple random walk to be transient, what can we say about its escape velocity? In dimensions $d \ge 3$, it was shown in~\citep{DE,ET} that a simple random walk on $\Z^d$ escapes at a rate of about $\sqrt n /\log^{c_d+o(1)} n$. What can be said in our context?
\begin{problem}
Fix $d \ge 2$.  What is the supremum of $\alpha$ such that for a random walk $\WW$ on $\Z^d$, we can almost surely choose a subgraph $H \subset \Z^d$ such that the walk escapes to infinity at rate at least $n^\alpha$? 
\end{problem}
As a first step towards this problem, it would already be interesting to know if we can get a random walk to escape with linear velocity in high dimensions.

Finally, a fundamental problem in this context, and one of our original motivations for treating the problem considered here, comes from the theory of universal traversal sequences. Call an infinite word $\mathcal{Z}$ on the alphabet $\{\xx, -\xx, \yy, -\yy \}$ a \emph{universal traversal sequence} for $\Z^2$ if $\mathcal{Z}_G$ is recurrent for every $G \subset \Z^2$. The following basic question raised by Spink~\citep{spink} remains wide open.
\begin{problem}
	Does there exist a universal traversal sequence for $\Z^2$?
\end{problem}
David and Tiba~\citep{mar} recently found deterministic constructions of traversal sequences handling a reasonably large class of (but not all) subgraphs of $\Z^2$. However, in general, the most efficient methods that we know of to construct universal traversal sequences all involve choosing a long enough traversal sequence at random; our main result rules out this standard construction on the square lattice. Either answer to the above existence question, positive or negative, would be very interesting.

\section*{Acknowledgements}
The first author was supported by SNF grant 175505, the fourth author was partially supported by NSF grant DMS-1800521, and the sixth author was supported by a Leverhulme Trust Research Fellowship.

The fourth author would also like to thank Yuval Peres, Perla Sousi and Peter Winkler for interesting discussions.

\section*{Note added in proof}
After this manuscript was completed and circulated, it was brought to our attention that the existence of exceptional subgraphs for a single simple random walk has independently been established by Balister, Bollob\'as, Leader and Walters~\citep{imre}.

\bibliographystyle{amsplain}
\bibliography{excep_rw}

\providecommand{\bysame}{\leavevmode\hbox to3em{\hrulefill}\thinspace}
\providecommand{\MR}{\relax\ifhmode\unskip\space\fi MR }
\providecommand{\MRhref}[2]{%
  \href{http://www.ams.org/mathscinet-getitem?mr=#1}{#2}
}
\providecommand{\href}[2]{#2}
\begin{thebibliography}{10}

\bibitem{miss}
O.~Adelman, K.~Burdzy, and R.~Pemantle, \emph{Sets avoided by {B}rownian
  motion}, Ann. Probab. \textbf{26} (1998), 429--464.

\bibitem{AH}
G.~Amir and C.~Hoffmann, \emph{A special set of exceptional times for dynamical
  random walk on $\mathbb{Z}^2$}, Electron. J. Probab. \textbf{13} (2008),
  1927--1951.

\bibitem{dyn}
I.~Benjamini, O.~H\"aggstr\"om, Yuval Peres, and J.~E. Steif, \emph{Which
  properties of a random sequence are dynamically sensitive?}, Ann. Probab.
  \textbf{31} (2003), 1--34.

\bibitem{C}
T.~K. Carne, \emph{A transmutation formula for {M}arkov chains}, Bull. Sci.
  Math. \textbf{109} (1985), 399--405.

\bibitem{mar}
S.~David and M.~Tiba, \emph{Solvability of mazes by blind robots}, Preprint,
  arXiv:1804.05439.

\bibitem{DE}
A.~Dvoretzky and P.~Erd\H{o}s, \emph{Some problems on random walk in space},
  Proceedings of the Second Berkeley Symposium on Mathematical Statistics and
  Probability, University of California Press, 1951, pp.~353--367.

\bibitem{ET}
P.~Erd\H{o}s and S.~J. Taylor, \emph{{Some problems concerning the structure of
  random walk paths}}, Acta Sci. Hung. \textbf{11} (1960), 137--162.

\bibitem{Fuk}
M.~Fukushima, \emph{Basic properties of {B}rownian motion and a capacity on the
  {W}iener space}, J. Math. Soc. Japan \textbf{36} (1984), 161--176.

\bibitem{sens}
C.~Hoffman, \emph{Recurrence of simple random walks on $\mathbb{Z}^2$ is
  dynamically sensitive}, ALEA \textbf{1} (2006), 35--45.

\bibitem{Kono}
N.~K\^ono, \emph{$4$-dimensional {B}rownian motion is recurrent with positive
  capacity}, Proc. Japan Acad. Ser. A Math. Sci. \textbf{60} (1984), 57--59.

\bibitem{imre}
I.~Leader, Personal communication, July 2018.

\bibitem{Chernoffcite}
M.~Mitzenmacher and E.~Upfal, \emph{Probability and computing}, Cambridge
  University Press, 2017.

\bibitem{VC}
R.~Peyre, \emph{A probabilistic approach to {C}arne's bound}, Potential Anal.
  \textbf{29} (2008), 17--36.

\bibitem{Pol}
G.~P\'olya, \emph{{{\"U}ber eine Aufgabe der Wahrscheinlichkeitsrechnung
  betreffend die Irrfahrt im Stra{\ss}ennetz}}, Math. Ann. \textbf{84} (1921),
  149--160.

\bibitem{spink}
H.~Spink, Personal communication, December 2014.

\bibitem{V}
N.~Th. Varopoulos, \emph{Long range estimates for {M}arkov chains}, Bull. Sci.
  Math. \textbf{109} (1985), 225--252.

\end{thebibliography}

\end{document}